 \documentclass[11pt]{article}

\usepackage{amsmath,amsthm,amssymb,xypic,stmaryrd, graphicx, mathtools}
\usepackage{hyperref}

\theoremstyle{plain}
    \newtheorem{theorem}{Theorem}[section]
    \newtheorem{lemma}[theorem]{Lemma}
    
    \newtheorem{proposition}[theorem]{Proposition}

 \theoremstyle{definition}
    \newtheorem{definition}[theorem]{Definition}
    
    \newtheorem{example}[theorem]{Example}

    \newtheorem{remark}[theorem]{Remark}

\theoremstyle{remark}

\numberwithin{equation}{section}

 \DeclareMathOperator{\Tr}{Tr}
 \DeclareMathOperator{\tr}{tr}
  
    \DeclareMathOperator{\id}{id}
\DeclareMathOperator{\im}{im}

\DeclareMathOperator{\Ad}{Ad}
\DeclareMathOperator{\ad}{ad}
\DeclareMathOperator{\ind}{index}

\DeclareMathOperator{\End}{End}

\DeclareMathOperator{\sgn}{sgn}
\DeclareMathOperator{\ch}{ch}

\DeclareMathOperator{\fp}{fp}

\DeclareMathOperator{\erf}{erf}

\DeclareMathOperator{\Spin}{Spin}

\DeclareMathOperator{\SO}{SO}

\DeclareMathOperator{\vol}{vol}

\begin{document}


\newcommand{\myemph}{\emph}

\newcommand{\Spinc}{\Spin^c}

    \newcommand{\R}{\mathbb{R}}
    \newcommand{\C}{\mathbb{C}}
    \newcommand{\N}{\mathbb{N}}
    \newcommand{\Z}{\mathbb{Z}}
    \newcommand{\Q}{\mathbb{Q}}
    \newcommand{\bT}{\mathbb{T}}
    \newcommand{\bP}{\mathbb{P}}

\newcommand{\g}{\mathfrak{g}}
\newcommand{\h}{\mathfrak{h}}
\newcommand{\p}{\mathfrak{p}}
\newcommand{\kg}{\mathfrak{g}}
\newcommand{\kt}{\mathfrak{t}}
\newcommand{\ka}{\mathfrak{a}}
\newcommand{\XX}{\mathfrak{X}}
\newcommand{\kh}{\mathfrak{h}}
\newcommand{\kp}{\mathfrak{p}}
\newcommand{\kk}{\mathfrak{k}}
\newcommand{\kn}{\mathfrak{n}}
\newcommand{\km}{\mathfrak{m}}
\newcommand{\kso}{\mathfrak{so}}
\newcommand{\ksu}{\mathfrak{su}}
\newcommand{\kspin}{\mathfrak{spin}}

\newcommand{\cA}{\mathcal{A}}
\newcommand{\cE}{\mathcal{E}}
\newcommand{\calL}{\mathcal{L}}
\newcommand{\calH}{\mathcal{H}}
\newcommand{\cO}{\mathcal{O}}
\newcommand{\cB}{\mathcal{B}}
\newcommand{\cK}{\mathcal{K}}
\newcommand{\cP}{\mathcal{P}}
\newcommand{\cN}{\mathcal{N}}
\newcommand{\calD}{\mathcal{D}}
\newcommand{\cC}{\mathcal{C}}
\newcommand{\calS}{\mathcal{S}}
\newcommand{\cM}{\mathcal{M}}
\newcommand{\cU}{\mathcal{U}}

\newcommand{\cCM}{\cC}
\newcommand{\PM}{P}
\newcommand{\DM}{D}
\newcommand{\LM}{L}
\newcommand{\vM}{v}

\newcommand{\sumGam}{\textstyle{\sum_{\Gamma}} }

\newcommand{\sigDg}{\sigma^D_g}

\newcommand{\Bigwedge}{\textstyle{\bigwedge}}

\newcommand{\ii}{\sqrt{-1}}

\newcommand{\Ubar}{\overline{U}}

\newcommand{\beq}[1]{\begin{equation} \label{#1}}
\newcommand{\eeq}{\end{equation}}

\newcommand{\mattwo}[4]{
\left( \begin{array}{cc}
#1 & #2 \\ #3 & #4
\end{array}
\right)
}

\title{A higher index and rapidly decaying kernels}

\author{
Hao Guo, 
Peter Hochs and 
Hang Wang
}
%
%
%

%

\date{\today}

\maketitle

\begin{abstract}
We construct an index of first-order, self-adjoint, elliptic differential operators in the $K$-theory of a Fr\'echet algebra of smooth kernels with faster than exponential off-diagonal decay. We show that this index can be represented by an idempotent involving heat operators. The rapid decay of the kernels in the algebra used is helpful in proving convergence of pairings with cyclic cocycles. Representing the index in terms of heat operators allows one to use heat kernel asymptotics  to compute such pairings. We give a link to von Neumann algebras and $L^2$-index theorems as an immediate application, and work out further applications in other papers.
\end{abstract}

\tableofcontents

\section{Introduction}

Let $X$ be a complete Riemannian manifold, and $E \to X$ a Hermitian vector bundle. Let $D$ be a first-order, self-adjoint, elliptic differential operator (such as a Dirac operator). The higher index 
\beq{eq Roe index intro}
\ind(D) \in K_*(C^*(X))
\eeq
of such operators $D$ in terms of the Roe algebra $C^*(X)$, including equivariant versions, has been used successfully in many places.  (See e.g.\  \cite{HR05I, HR05II, HR05III, KY06, PS14, Roe96, WY20, Yu00} for a very small selection.) To obtain relevant and computable numbers via pairings with traces and higher cyclic cocycles, however, it is useful to construct an index in the $K$-theory of an algebra of more rapidly decaying Schwartz kernels. 
 In this paper, we construct such an index.

 This index is used in  \cite{GHW25b} to obtain a $K$-theoretic index to which the traces that appear in the Selberg trace formula for cofinite-volume lattices in real semisimple Lie groups of real rank one \cite{OW78, Warner79} can be applied, refining the index theorem in \cite{BM83}. In ongoing work, we compute further pairings of the index in this paper with cyclic cocycles, using the rapid decay of the kernels in the algebra we use here to ensure convergence of these pairings.

We fix an algebra $S$ of differential operators on the exterior tensor product $E \boxtimes E^* \to X \times X$ with suitable properties, see page \pageref{page S} for details. We then consider a Fr\'echet algebra $\cA_S(E)$ of smooth sections of $E \boxtimes E^*$ that have faster off-diagonal decay than any exponential function of the Riemannian distance, and whose derivatives with respect to operators in $S$ also have this decay behaviour. In fact, we use a slight generalisation of this algebra, involving a proper map $\varphi$ from another manifold $Y$ to $X$. This is to obtain a link with the Casselman Schwartz algebra in \cite{GHW25b}, where $X = G/K$, for $G$ a connected, real semisimple Lie group and $K<G$ maximal compact, $Y = G$ and $\varphi$ is the quotient map. We will ignore the map $\varphi$ for now, see Definition \ref{def AX phi} for details.

The main technical result in this paper, see  Theorem \ref{thm multipliers}, is that under certain assumptions on $D$, for all $t>0$, the operator
\[
R_t := \frac{1-e^{-tD^2}}{D^2}D
\]
is a multiplier of $\cA_S(E)$, if a group $H$ acts on $X$ and $E$, preserving all structure, and such that $X/H$ is compact. The operator $D$ is a multiplier of $\cA_S(E)$ for more basic reasons, and the heat operator
\[
e^{-tD^2} = 1-R_t D
\]
lies in $\cA_S(E)$ itself. It follows that, if $D$ is odd with respect to a grading on $E$, we have a class
\beq{eq D K1 intro}
[D] \in K_1(\cM(\cA_S(E))/\cA_S(E)),
\eeq
where $\cM(\cA_S(E))$ is the multiplier algebra of $\cA_S(E)$. Applying the boundary map $\partial$ in the long exact sequence associated to the ideal $\cA_S(E) \subset \cM(\cA_S(E))$ to \eqref{eq D K1 intro}, we obtain a higher index
\beq{eq index intro}
\ind_S(D):= \partial [D] \in K_0(\cA_S(E)).
\eeq
Furthermore, this index
can be represented by a standard idempotent involving heat operators, so that one can use heat kernel asymptotics to compute its pairings with traces and other cyclic cocycles. We show that $\cA_S(E)$ embeds into $C^*(M)$, and that the map induced by this inclusion maps 
\eqref{eq index intro} to \eqref{eq Roe index intro}.

We similarly obtain an equivariant index
\beq{eq H index intro}
\ind_{S,H}(D):= \partial [D] \in K_0(\cA_S(E)^H),
\eeq
where the superscript $H$ denotes the subalgebra of $H$-invariant kernels.

The assumptions made hold in the important special case of Dirac operators on the space $G/K$ as above, where the operators in $S$ are given by the left and right action of the universal enveloping algebra of the Lie algebra of $G$. 

Under milder assumptions, we obtain the weaker result (Proposition \ref{prop index weak}) that the index \eqref{eq Roe index intro} can be represented by an idempotent in (the unitisation of) $\cA_S(E)$. This allows one to pair \eqref{eq Roe index intro} with cyclic cocycles on $\cA_S(E)$ to some extent, but if  the map 
\beq{eq incl K}
K_0(\cA_S(E)) \to K_0(C^*(M))
\eeq
 induced by the inclusion is not injective, then
 it is not clear that such a pairing is well-defined. A related point is that  the index \eqref{eq index intro} may contain more information than \eqref{eq Roe index intro}.
Bradd \cite{Bradd23} has shown that a $G$-equivariant version of \eqref{eq incl K} is bijective for $X = G/K$, but in general it is not clear if \eqref{eq incl K} is injective. Then the index \eqref{eq index intro} defined directly in terms of the algebra $\cA_S(E)$ is more useful for pairings with traces and cyclic cocycles.

One immediate application is a link between the equivariant  index \eqref{eq H index intro}, with $H$ replaced by a cofinite-volume, discrete subgroup $\Gamma <H$, to von Neumann algebras and $L^2$-index theorems. See Section \ref{sec vN}. As mentioned above, we work out further  applications in other papers.

\subsection*{Acknowledgements}

PH was supported by the Netherlands Organisation for Scientific Research NWO, through ENW-M grants  OCENW.M.21.176 and OCENW.M.23.063. HW was supported by the grants 23JC1401900, NSFC 12271165 and in part by Science and Technology Commission of Shanghai Municipality (No.\ 22DZ2229014).

\section{A higher index}\label{sec prelim}

\subsection{An algebra of exponentially decaying kernels}\label{sec prelim alg}

Consider a  Riemannian manifold $X$, with  Riemannian distance $d$ and Riemannian density $\mu$.
\begin{definition}\label{def unif exp}
The manifold $X$ has \emph{uniformly exponential growth} if there are $a_0,C>0$ such that  for all $x \in X$, 
\[
\int_X e^{-{a_0} d(x,x')}\, d\mu(x') \leq C.
\]
\end{definition}
If $x \in X$ and $r>0$, then we write $B_r(x)$ for the open Riemannian ball around $x$ of radius $r$.
\begin{example}\label{ex ue cocpt}
Suppose that a group $H$ acts isometrically on $X$, such that $X/H$ is compact. Let $x_0 \in X$ and $a_0>0$, and suppose that 
\[
\int_X e^{-a_0 d(x_0,x')}\, d\mu(x') 
\]
converges. Let $r>0$ be such that $H \cdot B_r(x_0, r) = X$. Then for all  $x \in X$, write $x = hx_1$, for $h \in H$ and $x_1 \in B_r(x_0)$. Then by  $H$-invariance of $d$ and $\mu$ and the triangle inequality, 
\[
\int_X e^{-{a_0} d(x,x')}\, d\mu(x') = 
\int_X e^{-{a_0} d(x_1,x')}\, d\mu(x')
\leq
e^{a_0r}  \int_X e^{-{a_0} d(x_0,x')}\, d\mu(x'). 
\]
So  $X$ has uniformly exponential growth. This example includes the case where $X = G/K$, for a connected Lie group $G$, compact subgroup $K<G$, and a $G$-invariant Riemannian metric on $X$.
\end{example}

We suppose from now on that $X$ has uniformly exponential growth. We consider a Hermitian vector bundle $E \to X$. A section $\kappa$ of  the exterior tensor product $E \boxtimes E^* \to X \times X$ has \emph{finite propagation} if there is an $r>0$ such that for all $x,x' \in X$ with $d(x,x')>r$, we have $\kappa(x,x')=0$. If two $L^{\infty}$-sections $\kappa, \lambda$ of $E \boxtimes E^*$ have finite propagation, then their composition $\kappa \lambda$ is well-defined as an $L^{\infty}$-section of $E \boxtimes E^*$ with finite propagation, by
\beq{eq compos kernels}
(\kappa \lambda)(x,x') = \int_X \kappa(x,x'')\lambda(x'', x')\, d\mu(x''),
\eeq
for all $x, x' \in X$.

 Fix an algebra $S$ of differential operators on $E \boxtimes E^*$\label{page S} such that
\begin{enumerate}
\item $S$ contains
the identity operator;
\item
for all $l \in \Z_{\geq 0}$, the space of order $l$ differential operators in $S$ is a finite-dimensional linear subspace of $S$;
\item \label{page S 3} for all $P \in S$, there are differential operators $Q$ and $R$ on $E$ such that $Q \otimes 1$ and $1\otimes R^*$ lie in  $S$ and $P = Q \otimes R^*$. Here the differential operator $R^*$ on $E^*$ is defined by the property that for all $\sigma \in \Gamma^{\infty}(E^*)$ and $s \in \Gamma^{\infty}_c(E)$,
\[
\int_X \langle (R^* \sigma)(x), s(x)\rangle\, d\mu(x) = \int_X \langle \sigma(x), (Rs)(x)\rangle\, d\mu(x).
\]
 \end{enumerate}
 The third point implies that for all $\kappa \in \Gamma^{\infty}(E \boxtimes E^*)$ with finite propagation, the section $P(\kappa)$ is the Schwartz kernel of the composition $Q \circ \kappa \circ R$. Here, and in the rest of this paper, we identify $\kappa$ with the operator it defines.
\begin{example}\label{ex AS GK}
Let $G$ be a Lie group, and let $V$ be a finite-dimensional, complex vector space. Consider the trivial vector bundle $E = G \times V \to G$. 
Let $R$ be the right regular representation of the universal enveloping algebra $\cU(\kg)$ on $C^{\infty}(G)$. For $P_1,P_2 \in \cU(\kg)$, let the operator $R(P_1) \otimes R(P_2)$ on $\Gamma^{\infty}(E \boxtimes E^*) = C^{\infty}(G \times G) \otimes \End(V)$ be defined by the action by $R(P_1)$ on the first copy of $G$, and the action by $R(P_2)$ on the second copy of $G$. 
Let
\[
S:= \{R(P_1) \otimes R(P_2); P_1,P_2 \in \cU(\kg)\}.
\]
Then $S$ satisfies the three conditions listed above. For the third, if $P = R(P_1) \otimes R(P_2)$, then we may take $Q = R(P_1)$ and $R = R(P_2)^*$.
\end{example}

Let $\Gamma_{S, \fp}^{\infty}(E \boxtimes E^*)$ be the space of  smooth sections $\kappa$ of $E \boxtimes E^* \to X\times X$ with finite propagation, such that $P\kappa$ is bounded for all $P \in S$. 
\begin{definition}\label{def AX}
For $a>0$ and $P \in S$, the seminorm $\|\cdot \|_{a, P}$ on $\Gamma_{S, \fp}^{\infty}(E \boxtimes E^*)$ is given by
\beq{eq seminorms ASE}
\|\kappa\|_{a, P}  = \sup_{x,x' \in X} e^{a d(x,x')}\|(P\kappa)(x,x') \|,
\eeq
for $\kappa \in \Gamma_{S, \fp}^{\infty}(E \boxtimes E^*)$.
We denote the completion of $\Gamma_{S, \fp}^{\infty}(E \boxtimes E^*)$  in the seminorms $\|\cdot \|_{a, P}$, for $a >0$ and $P \in S$, by $\cA_S(E)$.
\end{definition}

\begin{lemma}\label{lem AX Frechet alg}
The space $\cA_S(E)$ is a Fr\'echet algebra under the composition \eqref{eq compos kernels}.
\end{lemma}
Lemma \ref{lem AX Frechet alg} is proved in Subsection \ref{sec AS}.

Now let $H$ be a Lie group (possibly with infinitely many connected components, so $H$ may be discrete), acting isometrically on $X$ and such that $X/H$ is compact. Suppose that $E$ is an $H$-equivariant vector bundle, and that the action preserves the Hermitian metric on $E$. 
Consider the  action by $H \times H$ on $\Gamma^{\infty}(E \boxtimes E^*)$ given by
\beq{eq HxH action}
((h,h') \cdot \kappa)(x,x') := 
h\kappa(h^{-1}x, h'^{-1}x')h'^{-1}, 
\eeq
for $h,h' \in H$, $x,x' \in X$ and $\kappa \in \Gamma^{\infty}(E \boxtimes E^*)$. Suppose that the operators in $S$ commute with this action. Unless specified otherwise, by the action of $H$ on $\Gamma^{\infty}(E \boxtimes E^*)$, we  mean the restriction of the action \eqref{eq HxH action} to the diagonal. (But note the notation \eqref{eq gamma circ kappa}.)
%
\begin{lemma}\label{lem ASEH Frechet}
The action by $H$ on $\Gamma^{\infty}(E \boxtimes E^*)$ restricts to an isometric action on $\cA_S(E)$. 
The fixed-point set $\cA_S(E)^H$ of this action is a Fr\'echet algebra.
\end{lemma}
Lemma  \ref{lem ASEH Frechet} is proved in Subsection \ref{sec AS}.

Let $C^*(X; E)^{H}$ be the $H$-equivariant Roe algebra of $X$ defined on the $C_0(X)$-module $L^2(E)$. This is the closure in $\cB(L^2(E))$ of the algebra of operators that are $H$-equivariant and locally compact, and have finite propagation. The algebra $S=\C$ trivially satisfies the assumptions made above, so we have the algebra $\cA_{\C}(E)^{H}$. 
\begin{lemma}\label{lem Roe GK}
The algebra $\cA_{\C}(E)^{H}$ is contained in $C^*(X; E)^{H}$, and the inclusion map is continuous.
\end{lemma}
This lemma is proved in Subsection \ref{sec AS}. All algebras $\cA_S(E)^H$, for more general $S$, are included continuously into $\cA_{\C}(E)^{H}$, and hence into $C^*(X; E)^{H}$. 

We will use slight generalisations  of the algebras $\cA_S(E)$ and $\cA_S(E)^H$. Let $Y$ be another Riemannian manifold with uniformly exponential growth. Let $\varphi\colon Y \to X$ be a smooth map. If $\kappa \in \Gamma^{\infty}(E \boxtimes E^*)$, then we define $\varphi^*\kappa \in \Gamma^{\infty}(\varphi^*E \boxtimes \varphi^*E^*)$ by
\beq{eq phi star}
(\varphi^*\kappa)(y,y') = \kappa(\varphi(y), \varphi(y')) \in E_{\varphi(y)} \otimes E^*_{\varphi(y')} = 
(\varphi^*E \boxtimes \varphi^*E^*)_{(y,y')},
\eeq
for $y,y' \in Y$.
We assume that $\varphi$ is
\begin{enumerate}
\item surjective;
\item
a
 \emph{quasi-isometry} in the sense that there are $A, B>0$ such that for all $y,y' \in Y$,
\beq{eq phi quasi isom}
A^{-1} d_Y(y,y') - B \leq
d_X(\varphi(y), \varphi(y')) \leq A d_Y(y,y') + B, 
\eeq
where $d_X$ is the Riemannian distance on $X$, and $d_Y$ is the one on $Y$;
\item an \emph{algebra homomorphism} in the sense that for all $\kappa, \lambda \in \Gamma_{\fp}^{\infty}(E \boxtimes E^*)$, we have
\[
\varphi^*(\kappa \lambda) =  (\varphi^* \kappa) (\varphi^*\lambda).
\]
\end{enumerate}
In the third point, we note that the second point implies that $ (\varphi^* \kappa)$ and $(\varphi^*\lambda)$ have finite propagation, so their composition is defined.

Let $S$ be an algebra of differential operators as above, but for the vector bundle $\varphi^*E \to Y$ rather than for $E$ itself.
Let $\Gamma_{S, \varphi, \fp}^{\infty}(E \boxtimes E^*)$ be the space of  smooth sections $\kappa$ of $E \boxtimes E^* \to X\times X$ with finite propagation, such that $P \varphi^*\kappa$ is bounded for all $P \in S$. 
\begin{definition}\label{def AX phi}
For $a>0$ and $P \in S$, the seminorm $\|\cdot \|_{a, P}$ on $\Gamma_{S, \varphi, \fp}^{\infty}(E \boxtimes E^*)$ is given by
\beq{eq seminorms ASE phi}
\|\kappa\|_{a, P}  = \|\varphi^* \kappa\|_{a, P}
\eeq
for $\kappa \in \Gamma_{S, \fp}^{\infty}(E \boxtimes E^*)$, where the right hand side is defined as in \eqref{eq seminorms ASE}.
We denote the completion of $\Gamma_{S, \varphi, \fp}^{\infty}(E \boxtimes E^*)$  in the seminorms $\|\cdot \|_{a, P}$, for $a >0$ and $P \in S$, by $\cA_{S, \varphi}(E)$.
\end{definition}
In this definition, note that surjectivity of $\varphi$ implies injectivity of $\varphi^*$. Hence if all seminorms \eqref{eq seminorms ASE phi} are zero, then in particular $\|\kappa\|_{1,1}=0$, so $\kappa = 0$.

By the construction and assumptions, the expression \eqref{eq phi star} defines an isometric algebra homomorphism
\beq{eq phi star A}
\varphi^* \colon \cA_{S, \varphi}(E) \xrightarrow{\cong} \cA_{S}(\varphi^*E).
\eeq
Hence $\cA_{S, \varphi}(E) $ is a Fr\'echet algebra by Lemma \ref{lem AX Frechet alg}.

If the Lie group $H$ also acts on $Y$, and $\varphi$ is $H$-equivariant, then $\varphi^*E$ naturally becomes an $H$-equivariant vector bundle. The algebra isomorphism \ref{eq phi star A} is now $H$-equivariant, so $\cA_{S, \varphi}(E)^H  \subset \cA_{S, \varphi}(E)$ is a Fr\'echet algebra by Lemma \ref{lem ASEH Frechet}, isomorphic to $\cA_{S}(\varphi^*E)^H$ via \eqref{eq phi star A}.

\begin{example}
If $Y=X$ and $\varphi = \id_X$ is the identity map, then $\cA_{S, \id_X}(E)= \cA_{S}(E)$ and $\cA_{S, \id_X}(E)^H= \cA_{S}(E)^H$. So Definition \ref{def AX phi} is a generalisation of Definition \ref{def AX}.
\end{example}

\begin{example}\label{ex ASqEG}
Consider the setting of Example \ref{ex AS GK}. Let $K<G$ be a compact subgroup, and suppose that $V$ carries a unitary representation of $K$. Let $X = G/K$, and let $E = G \times_K V \to X$ be  the quotient of $G \times V$ by the $K$-action given by
\[
k\cdot (g,v) = (gk^{-1}, k \cdot v),
\]
for $k \in K$, $g \in G$ and $v \in V$.  Let $q \colon G \to X$ be the quotient map. Then $q^*E = G \times V$, and $q$ satisfies all the assumptions on the map $\varphi$ above. For the third, we choose a left Haar measure $dg$ on $G$, and   define the measure $d(gK)$ on $G/K$ by
\beq{eq def dgK}
\int_{G/K} f(gK)\, d(gK) = \int_G (q^*f)(g)\, dg
\eeq
for all $f \in C_c(G/K)$. Then for all $\kappa, \lambda \in \Gamma_{\fp}^{\infty}(E \boxtimes E^*)$, and all $g,g' \in G$,
\begin{multline*}
q^*(\kappa \lambda)(g,g') = \int_{G/K} \kappa(gK, g''K)\lambda(g''K, g'K)\, d(gK)\\
=
 \int_{G} \kappa(gK, g''K)\lambda(g''K, g'K)\, dg = (q^*\kappa) (q^* \lambda)(g,g'). 
\end{multline*}

Consider the actions by $G$ on $X$, $E$ and $G$ induced by left multiplication. Then the operators $R(P_1) \otimes R(P_2)$ in Example \ref{ex AS GK} are $G\times G$-equivariant for all $P_1, P_2 \in \cU(\kg)$.
And $q$ is equivariant for this action, so we obtain the Fr\'echet algebra $\cA_{S, q}(E)^G$.

Note that the operators $R(P_1) \otimes R(P_2)$  do not preserve the subspace 
\[
\Gamma^{\infty}(E \boxtimes E^*) = (C^{\infty}(G \times G) \otimes \End(V))^{K \times K} \subset C^{\infty}(G \times G) \otimes \End(V), 
\]
which is why the algebra $\cA_{S, q}(E)^G$ is used in this case.
\end{example} 

In \cite{GHW25b}, we will use the algebra $\cA_{S, q}(E)^G$  in Example \ref{ex ASqEG}, and link it to Casselman's Schwartz algebra of functions on Lie groups.


\subsection{Construction of the index}\label{sec prelim index}
%
%
%

Let $D$ be an $H$-equivariant, first-order, elliptic, self-adjoint differential operator on $E$. We make the following three assumptions.
\begin{enumerate}
\item There are $Q,R \in S$ such that   $\varphi^* \circ (D^2\otimes 1) = Q \circ \varphi^*$ and $\varphi^* \circ (1 \otimes (D^2)^*) = R \circ \varphi^*$.
\item  For operators $T \in \cB(L^2(E))$ with well-defined, bounded commutators with $D^2$, we write $\ad(D^2)(T) = [D^2, T]$. We assume that for all $f \in C^{\infty}_c(X)$ and $j \geq 0$, the repeated commutator $\ad(D^2)^j(f)$ is a well-defined, bounded operator on $L^2(E)$. Furthermore, if $Z$ is any smooth manifold, and $\tilde f \in C^{\infty}(Z \times X)$ is such that $\tilde f(z, \relbar) \in C^{\infty}_c(X)$ for all $z \in Z$, then we assume that for all $j \geq 0$, the operator norm of $\ad(D^2)^j(\tilde f(z, \relbar))$ depends continuously on $z$.  
\item For all $P \in S$, the compositions $P \circ \varphi^* \circ (D \otimes 1)$ and $P \circ \varphi^* \circ (1 \otimes D)$ are finite sums of compositions of the form $A \circ Q \circ \varphi^*$, for 
bounded endomorphisms $A$ of $\varphi^*E \boxtimes \varphi^* E^*$ and  $Q \in S$.
\end{enumerate}
\begin{example}\label{ex cond D GK}
Consider the setting of Example \ref{ex ASqEG}. Let $\kg = \kk \oplus \kp$ be a Cartan decomposition. The adjoint representation $\Ad\colon K \to \SO(\kp)$ lifts to  double covers $\widetilde{K} \to \Spin(\kp)$. In this way, the standard representation $\Delta_{\kp}$ of $\Spin(\kp)$ may be viewed as a representation of $\widetilde{K}$. 
 Let $W$ be a finite-dimensional, unitary representation of $\widetilde{K}$ such that $V = \Delta_{\kp} \otimes W$ descends to a representation of $K$.
 
 Let $\{X_1, \ldots, X_r\}$ be a orthonormal basis of $\kp$ with respect to a $K$-invariant inner product. Let $c\colon \kp \to \End(\Delta_{\kp})$ be the Clifford action. Then we have the Dirac operator
 \beq{eq Dirac GK}
 D := \sum_{j=1}^r R(X_j) \otimes c(X_j) \otimes 1_W
 \eeq
 on $\bigl( C^{\infty}(G) \otimes \Delta_{\kp} \otimes W\bigr)^K \cong \Gamma^{\infty}(G \times_K V)$. This operator satisfies the assumptions listed above, for $H = G$. 
 
 Indeed, by Proposition 3.1 in \cite{Parthasarathy72}, 
 \[
 D^2 = R(\Omega) + b \quad \in \cU(\kg),
 \]
 where $\Omega \in \cU(\kg)$  is the Casimir element and $b$ is a scalar. So the first assumption holds, with $Q = R(\Omega) \otimes 1$ and $R = 1 \otimes R(\Omega)$.
The second assumption holds as well, because for all $f \in C^{\infty}_c(G)$,
 \[
 \ad(D^2)^j(f) = R( \Omega)^j(f)
 \]
 is multiplication by a compactly supported smooth function.
 
 For the third assumption, let $P = R(P_1) \otimes R(P_2) \in S$. 
Now $q^* \circ D = \widetilde{D} \circ q^*$, where $\widetilde{D}$ is given by the same expression \eqref{eq Dirac GK}. 
So
 \[
 P \circ q^* \circ (D \otimes 1) = \sum_{j=1}^r  (c(X_j) \otimes 1)\circ R(P_1 X_j) \otimes R(P_2) \circ q^*,
 \]
 which is of the desired form. The argument for $P \circ q^* \circ  (1\otimes D)$ is analogous.
\end{example}

We fix $t>0$ once and for all. 
Consider the operator 
\beq{eq def Rt}
R_t := \frac{1-e^{-tD^2}}{D^2}D \in \cB(L^2(E)).
\eeq
The main technical result in this paper is the following.
\begin{theorem}\label{thm multipliers}
The operators $D$ and $R_t$ are multipliers of $\cA_{S, \varphi}(E)$.
\end{theorem}
This result is proved at the end of Subsection \ref{sec pf mult}. (The case for $D$ is immediate from the assumptions, see Lemma \ref{lem D mult}, so the case for $R_t$ is the main step.) It has the following consequence.

Suppose that  $E$ has an $H$-invariant  $\Z/2\Z$ grading.
Suppose that   $D$ is  odd with respect to  this grading. We denote the restrictions of $D$ to even- and odd-graded sections by $D^+$ and $D^-$, respectively. Consider the boundary map 
\beq{eq maps K10}
\partial\colon 
K_1(\cM(\cA_{S, \varphi}(E)^H)/\cA_{S, \varphi}(E)^H) \to K_0(\cA_{S, \varphi}(E)^H),
\eeq
associated to the long exact sequence corresponding to the ideal $\cA_{S, \varphi}(E)^H$ in its algebra of multipliers $\cM(\cA_{S, \varphi}(E)^H)$. 

\begin{theorem}\label{thm def index}
The operator $D \in \cM(\cA_{S, \varphi}(E)^H)$ is invertible  modulo $\cA_{S, \varphi}(E)^H$. The image of the resulting class  $ [D] \in K_1(\cM(\cA_{S, \varphi}(E)^H)/\cA_{S, \varphi}(E)^H)$ under the map \eqref{eq maps K10} equals
\begin{multline}\label{eq index idempotent}
\partial [D] = \\
\left[
\begin{pmatrix}
e^{-tD^-D^+} & e^{-\frac{t}{2}D^-D^+} \frac{1-e^{-tD^-D^+}}{D^-D^+}D^-\\
e^{-\frac{t}{2}D^+D^-}D^+ & 1-e^{-tD^+D^-} 
\end{pmatrix}\right] - 
\left[
\begin{pmatrix}
0& 0\\
0 & 1
\end{pmatrix}\right]
\quad \in K_0(\cA_{S, \varphi}(E)^H).
\end{multline}
\end{theorem}
\begin{proof}
The operators $D$ and $R_t$ lie in $\cM(\cA_{S, \varphi}(E))$ by Theorem \ref{thm multipliers}. Since these operators are $H$-equivariant, they also lie in $\cM(\cA_{S, \varphi}(E)^H)$. 
We have 
\[
1- DR_t  = 1-R_tD = e^{-tD^2},
\]
which lies in $\cA_{S, \varphi}(E)^H$. Using the pseudo-inverse $R_t$ of $D$, one can show that the image of $[D] \in K_1(\cM(\cA_{S, \varphi}(E)^H)/\cA_{S, \varphi}(E)^H)$ under \eqref{eq maps K10} is \eqref{eq index idempotent}, see page 356 of \cite{CM90}.
\end{proof}

\begin{definition}\label{def H index}
The $H$-equivariant \emph{higher index} of $D$ with respect to $S$ and $\varphi$ is the class \eqref{eq index idempotent}. It is denoted by $\ind_{S,H}(D) \in K_0(\cA_{S, \varphi}(E)^H)$.
\end{definition}


\begin{remark}
By forgetting $H$-equivariance, we find that Theorem \ref{thm def index} is also true if $\cA_{S, \varphi}(E)^H$ is replaced by $\cA_{S, \varphi}(E)$. This leads to a higher index
\[
\ind_{S}(D) \in K_0(\cA_{S, \varphi}(E)),
\]
given by \eqref{eq index idempotent}.
\end{remark}

\begin{remark}\label{rem Roe index}
The Roe algebra index  is given by the same idempotents \eqref{eq index idempotent} as the $\ind_{S}(D)$. See (3.9) in \cite{HW2} and the references given there. 
By Lemma \ref{lem Roe GK}, we have continuous homomorphisms
\beq{eq AS Roe index}
\cA_{S, \varphi}(E) \xrightarrow[\cong]{\varphi^*} \cA_{S}(\varphi^*E) \hookrightarrow  \cA_{\C}(\varphi^*E) \hookrightarrow C^*(Y, \varphi^*E) \to C^*(Y) \xleftarrow{\varphi^*} C^*(X).
\eeq
Here $C^*(Y)$ and $C^*(X)$ are the Roe algebras of $Y$ and $X$ defined with respect to the ample $C_0(Y)$-module $L^2(\varphi^*E) \otimes l^2(\N)$ and the ample $C_0(X)$-module $L^2(E) \otimes l^2(\N)$, respectively.  Since $\varphi$ is a coarse equivalence, the right-most map in the above sequence induces an isomorphism on $K$-theory. We find that the map from $K_0(\cA_{S, \varphi}(E))$ to $K_0(C^*(X))$ induced by \eqref{eq AS Roe index} maps $\ind_{S}(D)$ to the Roe algebra index of $D$. Equivariant versions of this fact hold as well, though one should take care if $H$ is not discrete \cite{GHM21}.  
\end{remark}



Theorem \ref{thm def index}, and hence Definition \ref{def H index},  applies if $D$ satisfies the three conditions at the start of this subsection, which is true in the setting of Example \ref{ex cond D GK}. If $D$ does not necessarily satisfy these conditions, then a weaker result is still true. This may still be useful for defining pairings with cyclic cocycles, for example. We state a non-equivariant version of this result, but equivariant versions also hold. (Even in this non-equivariant case, we assume $D$ to be equivariant with respect to the cocompact action by $H$. This is used for example in the proof of Proposition \ref{prop kernel decay}.)
\begin{proposition}\label{prop index weak}
Let $D$ be an $H$-equivariant, first-order, elliptic, self-adjoint differential operator on $E$, but do not assume that it satisfies conditions 1--3 at the start of this subsection. Then the  Roe algebra index of $D$ in $K_0(C^*(X; E))$ can be represented as by the idempotents \eqref{eq index idempotent}, whose elements lie in $\cA_{S, \varphi}(E)$.
\end{proposition}
This proposition is proved at the end of Subsection \ref{sec Schwartz decay}.


\section{Exponentially decaying kernels}

In this section, we prove the results stated in Section \ref{sec prelim}. In Subsection \ref{sec AS}, we prove Lemmas \ref{lem AX Frechet alg}, \ref{lem ASEH Frechet} and \ref{lem Roe GK}. In Subsection \ref{sec sup kernel}, we obtain estimates for suprema of Schwartz kernels of certain operators defined by functional calculus applied to $D$. We apply these estimates in Subsection \ref{sec Schwartz decay} to show that for $f$ in a certain algebra $B(\R)$ of functions on $\R$, the operator $f(D)$ lies in $\cA_{S, \varphi}(E)^H$, see Proposition \ref{prop BR embed}. In Subsection \ref{sec pf mult}, we use results from Subsections \ref{sec AS}--\ref{sec Schwartz decay} to prove Theorem \ref{thm multipliers}.

\subsection{The algebra $\cA_S(E)$} \label{sec AS}

Consider the setting of Subsection \ref{sec prelim alg}.

\begin{proof}[Proof of Lemma \ref{lem AX Frechet alg}]
If $\kappa \in \Gamma_{S, \fp}(E \boxtimes E^*)$ and $\|\kappa\|_{a,P}=0$ for all $a>0$ and $P \in S$, then in particular $\|\kappa\|_{1,1}=0$, because $1 \in S$. Hence $\kappa = 0$. 

Using bases of the finite-dimensional spaces of order $l$ differential operators in $S$, we can form a countable subset of the set of seminorms \eqref{eq seminorms ASE} inducing the same topology as the set of all those seminorms. So
the space $\cA_S(E)$ is a Fr\'echet space.

Let $a_0$ and $C$ be as in Definition \ref{def unif exp}. 
Let $a \geq a_0$, and $P \in S$. Let $Q$ and $R$ be as in the third assumption on $S$ in Subsection \ref{sec prelim alg}.
Let $\kappa, \lambda \in \Gamma_{S, \fp}^{\infty}(E \boxtimes E^*)$. Then for all $x,y \in X$,
\[
(P(\kappa \lambda))(x,y) = \int_X (Q \circ \kappa)(x,z) (\lambda \circ R)(z,y)\, d\mu(z),
\]
where the differential operators $Q$ and $R$ may be taken inside the integral because the integrand has compact support.
The norm of the right hand side is smaller than or equal to
\beq{eq frechet 1}
\|\kappa\|_{2a, Q \otimes 1} \|\lambda\|_{a, 1 \otimes R}  \int_X e^{-2ad(x,z)} e^{-a d(z,y)}\, d\mu(z).
\eeq
And by the triangle inequality, 
\[
d(z,y) \geq d(x,y) - d(x,z).
\]
So the right hand side of \eqref{eq frechet 1} is at most equal to
\[
\|\kappa\|_{2a, Q \otimes 1} \|\lambda\|_{a, 1 \otimes R}    e^{-ad(x,y)} \int_X e^{-ad(x,z)} \, d\mu(z) \leq 
\|\kappa\|_{2a, Q \otimes 1} \|\lambda\|_{a, 1 \otimes R}    e^{-ad(x,y)} C.
\]
Here we used that $a \geq a_0$. 
We conclude that
\[
\|\kappa \lambda\|_{a, P}  \leq C \|\kappa\|_{2a, Q \otimes 1} \|\lambda\|_{a, 1 \otimes R}.
\]
If $a<a_0$, then 
\[
\|\kappa \lambda\|_{a, P}  \leq  \|\kappa \lambda\|_{a_0, P}  \leq C \|\kappa\|_{2a_0, Q \otimes 1} \|\lambda\|_{a_0, 1 \otimes R}.
\]
\end{proof}

\begin{proof}[Proof of Lemma \ref{lem ASEH Frechet}]
If 
$\kappa \in \cA_S(E)$, $h \in H$, $a>0$ and $P \in S$, then by $H$-equivariance of $P$, and $H$-invariance of the metric on $E$, 
\[
\|h\cdot \kappa\|_{a,P} = \|\kappa\|_{a, P}.
 \]
 So the action by $H$ on $\Gamma^{\infty}(E \boxtimes E^*)$ indeed restricts to   
  an isometric action by $H$  on $\cA_S(E)$. Because  the action is continuous and $\cA_S(E)$  is Hausdorff, the fixed point set of every element $h \in H$ is closed. Hence so is the intersection $\cA_S(E)^H$ of all these fixed-point sets.
\end{proof}

%

 Let $a_0,C>0$  be as in Definition \ref{def unif exp}.
%
\begin{lemma}\label{lem bdd GK}
The action by an element $\kappa \in \cA_{\C}(E)$ on $L^2(E)$ is a bounded operator, and
\beq{eq est Roe}
\|\kappa\|_{\cB(L^2(E))} \leq  C \|\kappa\|_{2a_0, 1}.
\eeq
\end{lemma}
\begin{proof}
Let  $\kappa \in \Gamma_{\C, \fp}^{\infty}(E)$. Let $s \in L^2(E)$. 
Then 
\begin{multline*}
\|\kappa \cdot s\|_{L^2(E)}^2  \leq \int_X \int_X \int_X \|\kappa(x,x')\| \| \kappa(x,x'') \| \|s(x')\| \|s(x'')\|\, d\mu(x')\, d\mu(x'')\, d\mu(x) \\
 \leq \| \kappa  \|_{2a_0, 1}^2 
\int_X \int_X \int_X e^{-2a_0 (d(x,x') + d(x,x''))} \|s(x')\| \|s(x'')\|\, d\mu(x')\, d\mu(x'')\, d\mu(x)\\
 \leq \| \kappa  \|_{2a_0, 1}^2 
\int_X \int_X \int_X e^{-a_0 (d(x,x') + d(x,x''))} e^{-a_0d(x',x'')}\|s(x')\| \|s(x'')\|\, d\mu(x')\, d\mu(x'')\, d\mu(x),
\end{multline*}
by the triangle inequality. Because the integrand is nonnegative, the integral on the right converges if and only if the following integral converges, and then these integrals are equal:
\beq{eq est Roe 1}
\|
\kappa  \|_{2a_0, 1}^2 
\int_X \int_X e^{-a_0d(x',x'')}\|s(x')\| \|s(x'')\|  \left( \int_X e^{-a_0 (d(x,x') + d(x,x''))} \, d\mu(x) \right)\,  d\mu(x')\, d\mu(x'').
\eeq
For fixed $x',x'' \in X$, the Cauchy--Schwartz inequality implies that
\[
\int_X e^{-a_0 (d(x,x') + d(x,x''))} \, d\mu(x)
  \leq C.
\]
So \eqref{eq est Roe 1} is bounded above by
\beq{eq est Roe 2b}
\| \kappa  \|_{2a_0, 1}^2 C
\int_X \int_X e^{-a_0d(x',x'')}\|s(x')\| \|s(x'')\|\,   d\mu(x')\, d\mu(x'').
\eeq

For every $x'' \in X$, the Cauchy--Schwartz inequality implies that
\[
 \int_X e^{-a_0d(x',x'')/2}\|s(x')\| \,   d\mu(x') \leq 
 \left( \int_X e^{-a_0d(x',x'')}\,   d\mu(x') \right)^{1/2} \|s\|_{L^2(E)} \leq C^{1/2}\|s\|_{L^2(E)}. 
\]
Applying this inequality twice, we find that  \eqref{eq est Roe 2b} is bounded above by
\[
\| \kappa  \|_{2a_0, 1}^2 C^2 \|s\|_{L^2(E)}^2.
\]
%
%
So \eqref{eq est Roe} follows. We also find that, by the Fubini--Tonelli theorem, all integrals that appeared converge absolutely.
\end{proof}

%

\begin{proof}[Proof of Lemma \ref{lem Roe GK}.]
The dense subalgebra $\Gamma_{\C, \fp}^{\infty}(E \boxtimes E^*)^{H}$ of $\cA_{\C}(E)^{H}$ consists of $H$-equivariant, locally compact operators with finite propagation. So the claim follows from Lemma \ref{lem bdd GK}.
\end{proof}

\subsection{Suprema of Schwartz kernels} \label{sec sup kernel}

Until the end of this section, we use the notation and assumptions in Subsection \ref{sec prelim index}. (But the assumptions 1--3 on $D$ made there will not be used until Subsection \ref{sec pf mult}.)

We will prove and use extensions of Lemmas 3.13 and 3.16 in \cite{CWXY19}.
Versions of the following fact seem to be well-known, compare  e.g.\ Proposition 5.9 in \cite{CGRS14}. We include a proof here because we were not immediately able to find a concise proof in the literature.

\begin{lemma}\label{lem HS}
For all $s> \dim(X)/2$, and all $f \in C_c(X)$, the operator $f(1+D^2)^{-s}$ on $L^2(E)$ is Hilbert--Schmidt.
\end{lemma}
\begin{proof}
Let $s>0$.
If $f$ is real-valued, then $f$ and  $(1+D^2)^{-s}$ are self-adjoint. Hence $f(1+D^2)^{-s}$ is Hilbert--Schmidt
if and only if $(1+D^2)^{-s}f$ is. If $f$ is not real-valued, then we reach the same conclusion by splitting $f$ up into real and imaginary parts. 

We will show that $(1+D^2)^{-s}f$ is Hilbert--Schmidt. Using the Mellin transform, we write
\[
(1+D^2)^{-s}f = \frac{1}{\Gamma(s)}\int_0^{\infty} t^{s-1} e^{-t} e^{-tD^2}f\, dt.
\]
Let $\kappa_t$ be the Schwartz kernel of $e^{-tD^2}$. 
Because all structures are preserved by the cocompact action by $H$, the manifold $M$ and the bundle $E$ have bounded geometry. So there are  $a, C_1>0$ such that
 for all $t \in (0,1]$ and $x,y \in X$,
\beq{eq heat decay}
\|\kappa_t(x,y)\|\leq C_1 t^{-\dim(X)/2} e^{-ad(x,y)^2/t}.
\eeq
See Proposition 5.2 in \cite{CGRS14}.
Because $X$ has uniformly exponential growth, it follows that the Schwartz kernel of $e^{-tD^2}f$ is square-integrable, so $e^{-tD^2}f$ is Hilbert--Schmidt. 

Now  for all $t>1$,
\[
\| e^{-tD^2}f\|_{\calL^2(L^2(E))}  = \| e^{-(t-1)D^2} e^{-D^2}f\|_{\calL^2(L^2(E))} \leq \|e^{-D^2}f\|_{\calL^2(L^2(E))}, 
\]
because $e^{-(t-1)D^2}$ has operator norm at most $1$. And for all $t \leq 1$ and $x,y \in X$, we have $e^{-ad(x,y)^2/t} \leq e^{-ad(x,y)^2}$. So then, by \eqref{eq heat decay}, 
\[
\| e^{-tD^2}f\|_{\calL^2(L^2(E))}^2 \leq  C_1^2 t^{-\dim(X)} \int_{X \times X} e^{-2ad(x,x')^2} |f(x')|^2 d\mu(x) d\mu(x').
\]

We find that for all $s> \dim(X)/2$, 
\begin{multline*}
\| (1+D^2)^{-s}f  \|_{\calL^2(L^2(E))} \leq \\ C_1 \left(\int_{X \times X} e^{-ad(x,x')^2} f(x') d\mu(x) d\mu(x') \right)^{1/2} \frac{1}{\Gamma(s)} \int_0^1 t^{s-\dim(X)/2 - 1} e^{-t}\, dt \\
+ \| e^{-D^2}f\|_{\calL^2(L^2(E))}
\int_1^{\infty} t^{s-1} e^{-t}  \, dt.
\end{multline*}
Both terms on the right converge.
\end{proof}

On a relatively compact open subset of a manifold, Sobolev norms of sections of a vector bundle can be defined unambiguously, up to equivalence. We denote the resulting $L^2$-Sobolev space of  order $k$ by $W^k$.

\begin{lemma}\label{lem phi Sob}
Let $X$ and $Y$ be smooth manifolds, and $\varphi\colon Y \to X$ smooth and proper. Let $E \to X$ be a Hermitian vector bundle. Let $U \subset X$ be open and relatively compact. For $s \in \Gamma^{\infty}(E|_U)$, define $\varphi^*s \in \Gamma^{\infty}((\varphi^*E)|_{\varphi^{-1}(U)})$ by
\[
\varphi^*s(y) = s(\varphi(y)),
\]
for $y \in \varphi^{-1}(U)$. For all $k \in \Z_{\geq 0}$, this defines a bounded map
\[
\varphi^*\colon W^k(E|_U) \to W^k((\varphi^*E)|_{\varphi^{-1}(U)}).
\]
\end{lemma}
\begin{proof}
By relative compactness of $U$, the statement reduces to the case where $U$ admits local coordinates $x^1, \ldots, x^m$ for $X$ and a trivialisation of  $E$. Then with respect to local coordinates $y^1, \ldots, y^n$ for $Y$ in an open subset  $V \subset \varphi^{-1}(U)$, the chain rule implies that for all multi-indices $\alpha \in \Z_{\geq 0}^n$, there are functions $F^{\beta}_{\alpha} \in C^{\infty}(V)$ (defined in terms of $\varphi$ and its partial derivatives) such that for all $s \in \Gamma^{\infty}(E|_U) = C^{\infty}(U) \otimes \C^r$,
\[
\frac{\partial^{\alpha} (\varphi^*s|_V)}{\partial y^{\alpha}} = \sum_{\beta \in \Z_{\geq 0}^m; |\beta| \leq |\alpha|} F^{\beta}_{\alpha} \frac{\partial^{\beta} s}{\partial x^{\beta}}  \circ \varphi|_V.
\]
Because $\varphi$ is  proper,  the set $\varphi^{-1}(U)$ is relatively compact. So it can be covered by finitely many coordinate neighbourhoods $V$, and the functions $F^{\alpha}_{\beta}$ on these sets are bounded. The Jacobian of $\varphi$ is bounded on $\varphi^{-1}(U)$, hence the claim follows. 
\end{proof}

Fix  integers $p> \dim(X)/2$ and $l \geq 0$.  
Let $\cB_l(L^2(E))^H$ be the space of 
$T \in \cB(L^2(E))^H$ such that for all $j,k \in \Z_{\geq 0}$ with $j+k \leq \dim(X) +2p+ l+ 1$, the composition $D^j T D^k$ defines a bounded operator on $L^2(E)$. For $T \in \cB_l(L^2(E))^H$ and $r \leq \dim(X) + 2p+l+ 1$, we write
\[
\|T\|_{r} := \max \left\{\|D^j T D^k\|_{\cB(L^2(E))}; j,k \in \Z_{\geq 0}, j+k \leq r\right\}.
\]

Let $S_l$ be a finite-dimensional vector space of $H \times H$-equivariant differential operators on $E \boxtimes E^*$ of orders at most $l$. Let $\varphi\colon Y \to X$ be as in Definition \ref{def AX phi}. We will denote the Schwartz kernel of an operator $T$ on $L^2(E)$, provided it has one, by $\kappa_T$. 
\begin{lemma}\label{lem sup kernel}
For all $P \in S_l$, there is a 
%
%
$C_P>0$ such that for all 
$T \in \cB_l(L^2(E))^H$, 
the section $P( \varphi^*\kappa_{T})$ is continuous and bounded, with $\sup$-norm at most
$
C_P  \|T\|_{\dim(X) +2p+ l+1}.
$ 
\end{lemma}
\begin{proof}
We generalise some arguments in the proof of Lemma 3.13 in \cite{CWXY19}.

Let $\gamma$ be the grading operator on $E$, equal to $1$ on the even part of $E$ and $-1$ on the odd part.
%
We write $D$ for the operator on $\Gamma^{\infty}(E^*)$ corresponding to $D$ via the identification $E^* \cong E$ defined by the metric. Then the operator
\[
D \otimes 1 + \gamma \otimes D
\]
on $\Gamma^{\infty}(E \boxtimes E^*)$ is elliptic.

Let $f \in C_c(X)$,   $n \leq \dim(X)+1$ and $T \in \cB_l(L^2(E))^H$. The operator $f(1+D^2)^{-p}$ is Hilbert--Schmidt by Lemma \ref{lem HS}, so 
\begin{multline}\label{eq sup kernel 1}
\|f (D \otimes 1 + \gamma \otimes D)^n  \kappa_{T}\|_{L^2(E \boxtimes E^*)}
\leq \sum_{j=0}^n {n\choose j} \|fD^j T D^{n-j}\|_{\calL^2(L^2(E))} \\
 \leq  \|f (1+D^2)^{-p}\|_{\calL^2(L^2(E))}\sum_{j=0}^n {n\choose j} 
 \| (1+D^2)^p D^jT D^{n-j}\|_{\cB(L^2(E))}.
\end{multline}
There is a constant $C_1>0$
 (depending on $f$, $D$ and $p$) such that the right hand side of \eqref{eq sup kernel 1} is less than or equal to 
$
C_1
\|T\|_{2p+n}.
$
In particular, the section $f (D \otimes 1 + \gamma \otimes D)^n  \kappa_{T}$ is indeed square-integrable.

Now let $U \subset X$ be open and relatively compact. Let $f \in C_c(X)$ be such that $f|_U \equiv 1$. Then by a Sobolev embedding theorem,
relative compactness of $U$, Lemma \ref{lem phi Sob} and finite-dimensionality of $S_l$,  there are constants $C_2, C_3>0$ (depending on $U$, $D$ and $P$), such that for all $\kappa \in \Gamma^{\infty}(E \boxtimes E^*)$,
\[
\begin{split}
\sup_{x,x' \in U} \| P(\varphi^*\kappa)(x,x')\| &\leq C_2 
\|\varphi^* \kappa\|_{W^{\dim(X)+l+1}(\varphi^*E|_{\varphi^{-1}(U)} \boxtimes \varphi^*E^*|_{\varphi^{-1}(U)}  )  }\\
&
\leq C_3 \|\kappa \|_{W^{\dim(X)+l+1}(E|_U \boxtimes E^*|_U)}.
\end{split}
\]
By  ellipticity of $D \otimes 1 + \gamma \otimes D$, the right hand side is smaller than or equal to 
\[ 
C_4
\sup_{n \leq \dim(X)+l+1} 
\|(D \otimes 1 + \gamma \otimes D)^n  \kappa\|_{L^2(E|_U \boxtimes E^*|_U)},
\]
for some constant $C_4>0$.
In particular, we have for all $P \in S_l$,
\beq{eq sup kernel 2}
\begin{split}
\sup_{x,x' \in U} \|P(\varphi^* \kappa_{T})(x,x')\| &\leq C_4 \sup_{n \leq \dim(X)+l+1} 
\|(D \otimes 1 + \gamma \otimes D)^n  \kappa_{T}\|_{L^2(E|_U \boxtimes E^*|_U)}\\
&\leq
C_4 \sup_{n \leq \dim(X)+l+1} 
\| f (D \otimes 1 + \gamma \otimes D)^n  \kappa_{T}\|_{L^2(E\boxtimes E^*)}\\
&
 \leq C_1 C_4
\|T\|_{\dim(X)+2p+l+1},
\end{split}
\eeq
where we used \eqref{eq sup kernel 1} and the sentence below it for the last estimate.

Now choose\footnote{We don't need the action to be proper here. If $x_0 \in X$ and $q\colon X \to X/H$ is the quotient map, then the cover $\bigcup_{r>0}q(B_r(x_0))$ of the compact space $X/H$ has a finite subcover. So we can take $U = B_r(x_0)$ for $r$ large enough. \label{footnote not proper}} $U$ (still relatively compact) such that $H \cdot U = X$. 
Then  for all $P \in S_l$, the facts that $P$ is $H\times H$-equivariant and that $\varphi$ is $H$-equivariant imply that
\beq{eq sup kernel 3}
\begin{split}
\sup_{x,x' \in X} \| P(\varphi^*\kappa_{T})(x,x')\| &= \sup_{x,x' \in U, h,h' \in H} \|  P(\varphi^*\kappa_{T})(h^{-1}x,h'^{-1}x')\|\\
&= \sup_{x,x' \in U, h,h' \in H} \| P(\varphi^*\kappa_{h \circ T \circ h'})(x,x')\|.
%
%
%
%
\end{split}
\eeq
By \eqref{eq sup kernel 2}, the right hand side 
is smaller than or equal to
\[
C_1 C_4 
\|h\circ T \circ h'^{-1}\|_{\dim(X)+2p+l+1} = C_1 C_4 
\|T \|_{\dim(X)+2p+l+1}. 
\]
Here we used $H$-equivariance of $D$, and the fact that $H$  acts unitarily on $L^2(E)$.
\end{proof}

\begin{remark}\label{rem sup kernel D2}
By replacing the operator $D$ by $D^2$, and the elliptic operator $D \otimes 1 + \gamma \otimes D$ by the elliptic operator $D^2 \otimes 1 + 1 \otimes D^2$, we obtain a version of Lemma \ref{lem sup kernel} in which the $\sup$-norm of $P (\varphi^*\kappa_T)$ is bounded by the operator norms of operators of the form $D^{2j} T D^{2k}$. This is used in the proof of Lemma \ref{lem near diag}. 
\end{remark}

\subsection{Decay behaviour of Schwartz kernels}\label{sec Schwartz decay}

Let $\hat B(\R)$ be the completion of $C^{\infty}_c(\R)$ in the seminorms defined by
\beq{eq seminorms BR hat}
\|\psi \|_{a, k} := \sup_{\xi \in \R} e^{a|\xi|} |\psi^{(k)}(\xi)|,
\eeq
for $\psi \in C^{\infty}_c(\R)$, $a>0$ and $k \in \Z_{\geq 0}$. Then $\hat B(\R)$ lies inside the space  of Schwartz functions, hence so does its inverse Fourier transform $B(\R)$. We equip the latter with the seminorms given by
\beq{eq seminorms BR}
\|f\|_{a, k} := \|\hat f\|_{a, k}
\eeq
for $f \in B(\R)$, $a>0$ and $k \in \Z_{\geq 0}$.
\begin{lemma}\label{lem BR alg}
The space $B(\R)$ is a Fr\'echet algebra with respect to pointwise multiplication.
\end{lemma}
\begin{proof}
We show that $\hat B(\R)$ is a Fr\'echet algebra with respect to convolution, which implies the claim.
The topology on $\hat B(\R)$ is given by the countable subset of seminorms $\|\cdot \|_{a, k}$ for $a,k \in \Z_{\geq 0}$, so $\hat B(\R)$ is a Fr\'echet space.
Let $\psi_1, \psi_2 \in \hat B(\R)$,   $a>0$ and $k \in \Z_{\geq 0}$. Then
\[
\begin{split}
\|\psi_1 * \psi_2\|_{a,k} &=
\sup_{\xi \in \R} e^{a|\xi|} \bigl| ( \psi_1^{(k)} * \psi_2)(\xi)  \bigr|\\
& \leq \|\psi_1\|_{2a, k}\|\psi_2\|_{a,0} \sup_{\xi} e^{a|\xi|} \int_{\R} e^{-2a |\eta|} e^{-a|\xi -\eta|}\, d\eta\\
& \leq \|\psi_1\|_{2a, k}\|\psi_2\|_{a,0}  \int_{\R} e^{-a |\eta|} \, d\eta,
\end{split}
\]
because $|\xi| - |\eta| \leq |\xi - \eta|$ for all $\xi, \eta \in \R$. So convolution is continuous. 
\end{proof}

\begin{remark}
In Definition 3.17 in \cite{CWXY19}, subspaces of functions $\cA_{\Lambda, N} \subset C_0(\R)$ are introduced, depending on parameters $\Lambda>0$ and $N \in \Z_{\geq 0}$. The intersection of all these spaces is a Fr\'echet algebra continuously included into $B(\R)$, see the proof of Lemma 3.18 in \cite{CWXY19}. In our setting, we find it easier to work with the algebra $B(\R)$. 
\end{remark}

%
%

Consider the setting of Definition \ref{def AX phi} in the equivariant case, i.e.\ including the actions by $H$.
%
%

For  $f \in B(\R)$, $n \in \Z_{\geq 0}$ and $s \geq 0$, write
\[
F_{f,n}(s) := \sum_{k=0}^n \int_{\R \setminus [-s,s]} 
|\hat f ^{(k)}(\xi)|
\, d\xi.
\]
Let $c_D$ be the propagation speed of $D$, i.e.\ the maximum of the norm of the principal symbol of $D$ on unit cotangent vectors. (This is finite, and indeed a maximum, because $D$ commutes with the cocompact action by $H$.) The following estimate is an extension of Lemma 3.16 in \cite{CWXY19}. Recall that we fixed an integer $p > \dim(X)/2$.
\begin{proposition}\label{prop kernel decay}
%
Let $n := \dim(X) + 2p + l+ 1$, and let $f \in B(\R)$. 
\begin{enumerate}
\item[(a)]
The section $P(\varphi^*\kappa_{f(D)})$ is  continuous and bounded for all $P \in S_l$.
\item[(b)] Let $A$ and $B$ be as in \eqref{eq phi quasi isom}. For all  $\mu>1$, and all $P \in S_l$,  there is a $C_P>0$ such that 
 for all $y,y' \in Y$,
\beq{eq kernel decay}
\|P(\varphi^*\kappa_{f(D)})(y,y')\| \leq C_P F_{f, n}\left(  \frac{d(y,y')-AB}{A\mu c_D} \right).
\eeq
\end{enumerate}
\end{proposition}
\begin{proof}
The proof is a modification of the proof of Lemma 3.16 in \cite{CWXY19}. 

The function $\hat f^{(n)}$ is a Schwartz function, hence so is the function $x \mapsto x^n f(x)$. In particular, the latter function is bounded. 
So $f(D) \in \cB_l(L^2(E))^H$. Hence Lemma \ref{lem sup kernel} implies part (a).

To prove part (b), let $r>0$ and $\mu>1$. 
Let $\varphi \in C^{\infty}(\R)$ be such that $\varphi(\xi) = 1$ if $|\xi|\geq 1$, and $\varphi(\xi) = 0$ if $|\xi|\leq 1/\mu$. Define $\varphi_r \in C^{\infty}(\R)$ by $\varphi_r(\xi) := \varphi( \xi c_D/r)$. Let $g_r$ be the inverse Fourier transform of $\varphi_r \hat f$. 

As in the proof of Lemma 3.16 in \cite{CWXY19}, we have for all $j,k \in \Z_{\geq 0}$ with $j+k \leq n$, 
\[
\| D^j g_r(D) D^k\|_{\cB(L^2(E))} \leq \psi(r) F_{f, n}(r/\mu c_D),
\]
where
\[
\begin{split}
\psi(r) &:= \frac{1}{2\pi} \max\left\{  {n\choose j} \|\varphi_r^{(j)}\|_{\infty}; j=0, \ldots, n  \right\}\\
&= \frac{1}{2\pi} \max\left\{  {n\choose j} \left(\frac{c_D}{r} \right)^j; j=0, \ldots, n  \right\}.
\end{split}
\]
So $\psi$ is decreasing in $r$.
We find that $g_r(D) \in \cB_l(L^2(E))^H$, and  Lemma \ref{lem sup kernel} implies that for all $P \in S_l$, 
there is a $C_1>0$ such that for all $r>0$, the section
 $P(\kappa_{g_r(D)})$ is continuous and bounded, with $\sup$-norm
 \beq{eq est gr} 
 \|P(\varphi^*\kappa_{g_r(D)})\|_{\infty} \leq C_1 \psi(r) F_{f, n}(r/\mu c_{D}).
 \eeq

Using finite-propagation arguments (specifically, Propositions 10.3.1 and 10.3.5 in \cite{Higson00} and their proofs) and the fact that
 $\hat g_r(\xi) = \hat f(\xi)$ if $|\xi| \geq r/c_D$, one shows that for all $s,s' \in \Gamma^{\infty}_c(E)$ whose supports are at least a distance $r$ apart,
 \[
 (s, g_r(D)s')_{L^2(E)} =  (s, f(D)s')_{L^2(E)}. 
 \]
So for all $x,x' \in X$ with $d_X(x,x') \geq r$, we have 
$
\kappa_{g_r(D)}(x,x') = \kappa_{f(D)}(x,x'). 
$
 If $y,y' \in Y$ and $d_Y(y,y')\geq A(r+B)$, then $d_X(\varphi(y), \varphi(y')) \geq r$, so
\beq{eq kappa gr f}
\varphi^*\kappa_{g_r(D)}(y,y') = \varphi^*\kappa_{f(D)}(y,y'). 
\eeq
Hence for all $y,y' \in Y$ with $d_Y(y,y') > A(r+B)$, we have
 \beq{eq P kappa f g}
 P(\varphi^*\kappa_{f(D)})(y,y') = P(\varphi^*\kappa_{g_r(D) })(y,y').
 \eeq
 
Now let $y,y' \in Y$, and suppose that $d_Y(y,y')>A(B+1)$. Then \eqref{eq est gr} and \eqref{eq P kappa f g} imply that   for all $r>1$ such that $A(r+B)<d_Y(y,y')$, and all $P \in S_l$, 
\begin{multline*}
\| P(\varphi^*\kappa_{f(D)})(y,y')  \| = \|P(\varphi^*\kappa_{g_r(D) })(y,y')\| \\
\leq C_1 \psi(r) F_{f, n}(r/\mu c_D) \leq C_1 \psi(1) F_{f, n}(r/\mu c_D). 
\end{multline*}
Here we used the fact that $\psi$ is decreasing. Because this inequality holds for all $r \in(1, d_Y(y,y')/A - B)$,
it extends by continuity to 
\beq{eq est d bigger 1}
\| P(\varphi^* \kappa_{f(D)})(y,y')  \|  \leq C_1 \psi(1) F_{f, n}\left(  \frac{d(y,y')-AB}{A\mu c_D} \right), 
\eeq
for all $y,y' \in X$ with $d_Y(y,y')>A(B+1)$.

To obtain an estimate for $y,y' \in Y$ with $d_Y(y,y')\leq A(B+1)$, we first note that the function 
\beq{eq fn sup Sl}
(y,y') \mapsto \sup_{P \in S_l}\| P(\varphi^*\kappa_{f(D)})(y,y')  \|
\eeq
is invariant under the diagonal action by $H$, because $D$ and $\varphi$ are $H$-equivariant and $P$ is $H \times H$ equivariant.
%
%
Because $X/H$ is compact, and $\varphi$ is $H$-equivariant and proper, the quotient $Y/H$ is compact as well. Hence
so is the quotient of the set  $\{(y,y') \in Y\times Y; d_Y(y,y') \leq A(B+1)\}$ by $H$. So the  $H$-invariant function \eqref{eq fn sup Sl} is bounded on the latter set.
Together with the estimate \eqref{eq est d bigger 1} for the case where $d(x,y)>A(B+1)$, and the fact that $F_{f,n}$ is a decreasing function, this implies part (b).
%
%
\end{proof}

\begin{proposition}\label{prop BR embed}
For all $f \in B(\R)$, the Schwartz kernel $\kappa_{f(D)}$ of the operator $f(D)$ lies in $\cA_{S, \varphi}(E)^H$. In fact,  the map $f \mapsto \kappa_{f(D)}$ is a continuous, injective $*$-homomorphism from $B(\R)$ to $\cA_{S, \varphi}(E)^H$.
\end{proposition}
\begin{proof}
Let $f \in B(\R)$. Then for all $a,s>0$, and $n \in \Z_{\geq 0}$,
\beq{eq est Ff}
F_{f, n}(s) \leq \frac{2e^{-as}}{a} \sum_{k=0}^n \|f\|_{a, k}. 
\eeq
By \eqref{eq est Ff}, Proposition \ref{prop kernel decay} (with $\mu = 2$)  implies that for all $P \in S_l$,  there  are $n \in \Z_{\geq 0}$ and $C_P>0$ such that 
for all $a,a'>0$,
\[
\begin{split}
\| \kappa_{f(D)}\|_{a, P} &\leq C_P \sup_{y,y' \in Y} e^{ad_Y(y,y')} F_{f, n}\left( \frac{d_Y(y,y')-AB}{2Ac_D}\right)\\
&\leq 
\frac{2C_P}{a'} e^{a'B/2c_D} 
\sup_{y,y' \in Y}
e^{(a-a'/2Ac_D)d_Y(y,y')}
\sum_{k=0}^n \|f\|_{a', k} . 
\end{split}
\]
For $a' = 2A c_D a$, the right hand side equals 
\[
\frac{C_P}{A c_D a} e^{aAB} \sum_{k=0}^n \|f\|_{2A c_D a, k}. 
\]
So the map $f \mapsto \kappa_{f(D)}$ is continuous.
 It is clearly an injective $*$-homomorphism.
\end{proof}

\begin{proof}[Proof of Proposition \ref{prop index weak}]
The Roe algebra index of $D$ can be represented as \eqref{eq index idempotent}, as in Remark \ref{rem Roe index}. The claim is therefore that the entries of the matrices in \eqref{eq index idempotent} lie in $\cA_{S, \varphi}(E)$. By Proposition \ref{prop BR embed}, this is the case if the following functions lie in $B(\R)$:
\[
\begin{split}
f_1(x) &= e^{-tx^2};\\
f_2(x) &= x e^{-tx^2};\\
f_3(x) &= e^{-tx^2/2} \frac{1-e^{-tx^2}}{x}.
\end{split}
\]
The functions $f_1$ and $f_2$ lie in $B(\R)$, because their Fourier transforms are Gaussians times a constant (for $f_1$) or a linear function (for $f_2$). 

Consider the function
\beq{eq ht}
h_t(x) = \frac{1-e^{-tx^2}}{x}.
\eeq
For all $\xi \in \R$, 
\beq{eq hat ht}
\widehat{h_t}(\xi) = i\pi \left( \erf\left( \frac{\xi}{2\sqrt{t}}\right) - \sgn(\xi) \right).
\eeq
This function decays faster than any exponential function. Let $a>0$, and let $C_a>0$ be such that for all $\xi \in \R$,
\[
|\widehat{h_t}(\xi) | \leq C_a e^{-a|\xi|}.
\]
Let $\varphi \in C^{\infty}_c(\R)$ and $k \in \Z_{\geq} 0$.
Analogously to the proof of Lemma \ref{lem BR alg}, we find that 
\[
\|\varphi * \hat h_t\|_{a,k} \leq \|\varphi\|_{2a, k} C_a \int_{\R} e^{-a|\eta|}\, d\eta.
\]
It follows that convolution by $\hat h_t$ extends to a continuous operator on $\hat B(\R)$. Hence pointwise multiplication by $h_t$  is a continuous operator on $B(\R)$.
Because the function $x \mapsto e^{-tx^2/2}$ lies in $B(\R)$, so does $f_3$.
\end{proof}

\subsection{Multipliers of $\cA_{S, \varphi}(E)^H$}\label{sec pf mult}

We use  results from the preceding subsections to prove Theorem \ref{thm multipliers}. For $D$, this theorem follows from the assumptions.
\begin{lemma}\label{lem D mult}
The operator $D$ is a multiplier of $\cA_{S, \varphi}(E)^H$.
\end{lemma}
\begin{proof}
Let $\kappa \in \cA_{S, \varphi}(E)^H$, $P \in S$ and $a>0$. Then  by the third assumption on $D$ in Subsection \ref{sec prelim index},
\begin{multline*}
\sup_{x,x' \in X} e^{ad(x,x')} \|P\varphi^*((D\otimes 1)\kappa) (x,x')\| \\
\leq C' 
\sup_{x,x' \in X} e^{ad(x,x')}\sum_{j=1}^n \|(P_j \kappa) (x,x')\| =C' \sum_{j=1}^n \|\kappa\|_{a, P_j},
\end{multline*}
for $C'>0$ and $P_1, \ldots, P_n \in S$ independent of $\kappa$, $x$ and $x'$. 
We find that $D \circ \kappa \in \cA_{S, \varphi}(E)^H$. By analogous arguments, we also have $\kappa \circ D \in \cA_{S, \varphi}(E)^H$. 
\end{proof}

We now turn to a proof that $R_t$ multiplies $\cA_{S, \varphi}(E)^H$. We will deduce this from a more general result, Proposition \ref{prop f(D) multiplier} below. In the proof of this result, we use certain functions to split estimates into near-diagonal and off-diagonal parts.
\begin{lemma}\label{lem chi x}
For every $x \in X$, there is a function $\chi_x \in C^{\infty}_c(X)$, and for every $j \in \Z_{\geq 0}$, there is a $C_j>0$ (independent of $x$), such that 
\begin{enumerate}
\item for all $x \in X$, the function $\chi_x$ is supported in the ball $B_2(x)$, and constant $1$ on $B_1(x)$;
\item for all $x \in X$ and $j \in \Z_{\geq 0}$, the operator $\ad(D^2)^j(\chi_x)$ on $L^2(E)$ is bounded, with norm at most $C_j$.
\end{enumerate}
\end{lemma}
\begin{proof}
Let $\chi \in C^{\infty}_c(\R)$ be such that $\chi(t) = 1$ for $t \in [0,1]$ and $\chi(t) = 0$ for $t\geq 2$. For $x \in X$, define $\chi_x \in C^{\infty}_c(X)$ by
\[
\chi_x(x') = \chi(d(x,x')),
\]
for $x' \in X$. These functions clearly satisfy the first condition.

Let $Z \subset X$ be a compact set such that $H \cdot Z = X$; this exists since $X/H$ is compact. By compactness of $Z$ and continuity of the operator  norm of $\ad(D^2)^j(\chi_x)$ in $x$ (see the second assumption on $D$ in Subsection \ref{sec prelim index}), there is a $C_j>0$ such that for all $x \in Z$, 
\[
\|\ad(D^2)^j(\chi_x)\|_{\cB(L^2(E))} \leq C_j.
\]
Because $D$ is $H$-equivariant and $H$ acts unitarily on $L^2(E)$, this inequality extends to all $x \in X$.
\end{proof}

Let $f \in L^{\infty}(\R)$ be such that for all $\varepsilon>0$, $f$ has a smooth Fourier transform $\hat f$ on $\R \setminus [-\varepsilon, \varepsilon]$, and for all $a>0$ and $k \in \Z_{\geq 0}$, 
\[
\sup_{|\xi|> \varepsilon} e^{a|\xi|} |\hat f^{(k)}(\xi)|
\]
is finite. (The arguments below, in particular the proof of Lemma \ref{lem off-diag}, can be modified so that it is sufficient that there exists an $\varepsilon>0$ for which this is true, but we will not need this generalisation.)

Let $\kappa \in \cA_{S, \varphi}(E)^H$. We will show that $f(D) \circ \kappa \in \cA_{S, \varphi}(E)^H$ via estimates away from  the diagonal (Lemma \ref{lem off-diag}) and near the diagonal (Lemma \ref{lem near diag}). The estimates for 
$\kappa \circ f(D)$ are analogous.

Let $P  \in S$, and let $Q$ and $R$ be as in the third assumption on $S$ in Subsection \ref{sec prelim alg}. 
 Let $\varphi_{1/2} \in C^{\infty}(\R)$ be as in the proof of Proposition \ref{prop kernel decay} (with $r=1/2$). As in the same proof, let $g_{1/2}$ be the inverse Fourier transform of $\varphi_{1/2} \hat f$. 
Let $C, a_0>0$ be as in Definition \ref{def unif exp}.
\begin{lemma}\label{lem off-diag}
The operator $g_{1/2}(D)$ lies in $\cA_{S, \varphi}(E)^H$. 
For all $a\geq a_0$ and  $y,y' \in Y$,
\beq{eq off diag}
\|P\varphi^*(f(D)\circ (1-\chi_{\varphi(y)}) \circ \kappa )(y,y')\| \leq C e^{-ad_Y(y,y')} \|g_{1/2}(D)(1-\chi_{\varphi(y)})\|_{2a, Q \otimes 1} \|\kappa\|_{a, 1\otimes R}.
\eeq
\end{lemma}
\begin{proof}
The function $\varphi_{1/2} \hat f$ is smooth, and it, and all its derivatives, decays faster than any exponential function. So $g_{1/2} \in B(\R)$. By Proposition \ref{prop BR embed}, we have $g_{1/2}(D) \in \cA_{S, \varphi}(E)^H$.

As in the proof of \eqref{eq kappa gr f}, we have for all $x, x' \in X$ with $d_X(x,x') \geq 1/2$,
\[
\kappa_{g_{1/2}(D)}(x,x') = \kappa_{f(D)}(x,x').
\]
On the right hand side, we note that  the finite-propagation methods from \cite{Higson00}, in particular Proposition 10.3.5, used in the proof of Proposition \ref{prop kernel decay}, apply to bounded Borel functional calculus, so also to the distributional kernel of $f(D)$. 

Fix $y \in Y$. 
Let $x,x' \in X$, and suppose that $d_X(\varphi(y), x) \leq 1/2$. If $(1-\chi_{\varphi(y)})(x') \not=0$, then $d_X(\varphi(y), x') \geq 1$, so $d_X(x,x') \geq 1/2$. Hence
\[
\kappa_{g_{1/2}(D)} \circ (1-\chi_{\varphi(y)}) (x, x') =   \kappa_{f(D)} \circ (1-\chi_{\varphi(y)}) (x, x'). 
\]
(In particular, the kernel $\kappa_{f(D)} \circ (1-\chi_{\varphi(y)})$ is smooth on $B_{1/2}(\varphi(y)) \times X$.) 
So for all $y'' \in \varphi^{-1}(B_{1/2}(\varphi(y)))$ and $y' \in Y$, 
\[
\begin{split}
\varphi^*(\kappa_{g_{1/2}(D)} \circ (1-\chi_{\varphi(y)}) \circ \kappa) (y'', y')&= 
\int_X \kappa_{g_{1/2}(D)} \circ (1-\chi_{\varphi(y)}) (\varphi(y''), x)\kappa(x, \varphi(y'))\, d\mu(x)\\
&= 
\int_X \kappa_{f(D)} \circ (1-\chi_{\varphi(y)}) (\varphi(y''), x)\kappa(x, \varphi(y'))\, d\mu(x)\\
&= \varphi^*(\kappa_{f(D)}(D) \circ (1-\chi_{\varphi(y)}) \circ \kappa) (y'', y').
\end{split}
\]

Hence, by locality of $P$,
\[
P\varphi^*(f(D)\circ (1-\chi_{\varphi(y)}) \circ \kappa )(y,y') = 
P\varphi^*(g_{1/2}(D)\circ (1-\chi_{\varphi(y)}) \circ \kappa )(y,y').
\]
Now \eqref{eq off diag} follows as in the proof of Lemma \ref{lem AX Frechet alg}.
%

\end{proof}

\begin{lemma}\label{lem near diag}
For all $a\geq a_0$, there are $C_a>0$ and $P_1, \ldots, P_n \in S$ such that for all $y,y' \in Y$, 
\beq{eq near diag}
\|P\varphi^*(f(D)\circ \chi_{\varphi(y)}\circ \kappa )(y,y')\| 
\leq C_a e^{-ad_Y(y,y')} \sum_{j=1}^n \|\kappa\|_{a, P_j}.
\eeq
\end{lemma}
\begin{proof}
Let $y,y' \in Y$. Consider the operator
\[
T := f(D) \circ \chi_{\varphi(y)} \circ \kappa \circ \chi_{\varphi(y')}.
\]
Then $\varphi^* \kappa_T = \varphi^*(f(D)\circ \chi_{\varphi(y)}\circ \kappa)$ in a neighbourhood of $(y,y')$. 
So by locality of $P$, the left hand side of \eqref{eq near diag} equals
\beq{eq norm P phi kappa}
\|(P \varphi^* \kappa_T)(y,y')\|.
\eeq
We now apply a version of Lemma \ref{lem sup kernel}, with the operator $D$ replaced by $D^2$, and $D \otimes 1 + \gamma \otimes D$ by $D^2 \otimes 1 + 1 \otimes D^2$ (see Remark \ref{rem sup kernel D2}). This shows that \eqref{eq norm P phi kappa} is bounded above, uniformly in $y$ and $y'$,  by a constant times
\beq{eq near diag 6}
 \max_{k+l \leq m} \|D^{2l} T D^{2k}\|_{\cB(L^2(E))},
\eeq
for $m$ large enough, and provided we can show that the operators in \eqref{eq near diag 6} are bounded. Hence it remains to estimate the norms in \eqref{eq near diag 6}.

Let $k,l \in \Z_{\geq 0}$. 
Since $f(D)$ commutes with $D$, we have
\beq{eq near diag 2}
D^{2k} T D^{2l} = f(D) \circ D^{2k} \circ  \chi_{\varphi(y)} \circ \kappa \circ \chi_{\varphi(y')} \circ D^{2l}.
\eeq
Applying the equalities
\[
\begin{split}
D^{2k} \circ  \chi_{\varphi(y)} &= \sum_{r=0}^k {k \choose r} \ad(D^2)^r(\chi_{\varphi(y)}) \circ D^{2(k-r)};\\
\chi_{\varphi(y')}  \circ D^{2l}   &= \sum_{s=0}^l {l \choose s} (-1)^s  D^{2(l-s)} \ad(D^2)^s(\chi_{\varphi(y')}),
\end{split}
\]
and using the second assumption on $D$ in Subsection \ref{sec prelim index}, 
 we find that the operator on the right of \eqref{eq near diag 2} 
 is bounded, and that its operator norm is bounded by a constant times a finite sum of expressions of the form
\beq{eq near diag 3}
\|\ad(D^2)^r(\chi_{\varphi(y)}) \circ D^{2(k-r)} \circ \kappa \circ D^{2(l-s)} \circ \ad(D^{2})^s(\chi_{\varphi(y')})\|_{\cB(L^2(E))}.
\eeq

For $x = \varphi(y)$ and $x = \varphi(y')$, 
let $\psi_{x}  \in C^{\infty}_c(X)$ take values in $[0,1]$, and be constant $1$ on $B_2(x)$, and constant zero outside $B_3(x)$.
 Then  \eqref{eq near diag 3} equals
\[
\|\ad(D^2)^r(\chi_{\varphi(y)}) \circ \psi_{\varphi(y)} \circ D^{2(k-r)} \circ \kappa \circ D^{2(l-s)} \circ \psi_{\varphi(y')} \circ \ad(D^{2})^s(\chi_{\varphi(y')})\|_{\cB(L^2(E))}.
\]
By the properties of the functions $\chi_x$ in Lemma \ref{lem chi x}, the latter norm
is bounded above by a constant times
\[
\| \psi_{\varphi(y)} \circ D^{2(k-r)} \circ \kappa \circ D^{2(l-s)} \circ \psi_{\varphi(y')} \|_{\cB(L^2(E))}.
\]
By Lemma \ref{lem bdd GK}, this is less than or equal to
\begin{multline} \label{eq near diag 4}
C \sup_{x,x' \in X} e^{2a_0 d_X(x,x')} 
\| (\psi_{\varphi(y)} \circ D^{2(k-r)} \circ \kappa \circ D^{2(l-s)} \circ \psi_{\varphi(y')})(x,x')\|\\
\leq
C \sup_{d_X(x, \varphi(y)), d_X(x', \varphi(y')) \leq 3}  e^{2a_0 d_X(x,x')} 
\| (D^{2(k-r)} \circ \kappa \circ D^{2(l-s)})(x,x')\|.
\end{multline}

Let $Q,R \in S$ be as in the first assumption on $D$ in Subsection \ref{sec prelim index}, let $x, x' \in X$, and choose $\tilde y, \tilde y' \in Y$ such that $x = \varphi(\tilde y)$ and $x' = \varphi(\tilde y')$. Let $a>0$, and let $A$ be as in \eqref{eq phi quasi isom}. Then 
\beq{eq near diag 8}
\begin{split}
\| (D^{2(k-r)} \circ \kappa \circ D^{2(l-s)})(x,x')\|  &= \| \varphi^*(D^{2(k-r)} \circ \kappa \circ D^{2(l-s)})(\tilde y,\tilde y')\| \\
&= \|((Q^{k-r} \otimes R^{l-s}) \varphi^*\kappa)(\tilde y,\tilde y')\|\\
&\leq \|\kappa\|_{2a_0A+a, Q^{k-r} \otimes R^{l-s}} e^{-(2a_0A+a)d_Y(\tilde y, \tilde y')}.
\end{split}
\eeq
If we choose $\tilde y$ and $\tilde y'$ as above for all $x,x' \in X$, and use \eqref{eq phi quasi isom} and the  estimate \eqref{eq near diag 8}, then we find
 that the right hand side of \eqref{eq near diag 4} is less than or equal to a constant times
\beq{eq near diag 5}
\|\kappa\|_{2a_0A+a, Q^{k-r} \otimes R^{l-s}} 
\sup_{d_X(x, \varphi(y)), d_X(x', \varphi(y')) \leq 3}  
e^{-ad_Y(\tilde y, \tilde y')}.
\eeq
Finally, note that if $d_X(x, \varphi(y)) \leq 3$, then \eqref{eq phi quasi isom} implies that
$d_Y(\tilde y, y) \leq A(B+3)$, and similarly for $x'$, $y'$ and $\tilde y'$. So
\[
d_Y(\tilde y, \tilde y') \geq d_Y(y,y') - 2A(B+3).
\]
So \eqref{eq near diag 5} is bounded by a constant times
\beq{eq near diag 7}
\|\kappa\|_{2a_0A+a, Q^{k-r} \otimes R^{l-s}} e^{-ad_Y(y, y')}.
\eeq	
We find that \eqref{eq near diag 6} is bounded by a finite sum of constants times expressions like \eqref{eq near diag 7}, which 
are of the form on the right hand side of \eqref{eq near diag}. So the claim follows.
\end{proof}

\begin{proposition}\label{prop f(D) multiplier}
If $f$ is as  above Lemma \ref{lem off-diag}, then $f(D)$ is a multiplier of $\cA_{S, \varphi}(E)^H$.
\end{proposition}
\begin{proof}
Lemmas \ref{lem off-diag} and \ref{lem near diag} imply that for all $\kappa \in \cA_{S, \varphi}(E)^H$, $P \in S$ and $a>0$, there is a $C'>0$ such that for all  $y,y' \in Y$,
\[
\|P \varphi^*(f(D) \circ \kappa)(y,y')\| \leq C' e^{-ad_Y(y,y')}.
\]
So $f(D) \circ\kappa \in \cA_{S, \varphi}(E)^H$. Using analogous estimates, one shows that $\kappa \circ f(D) \in \cA_{S, \varphi}(E)^H$.
\end{proof}

\begin{proof}[Proof of Theorem \ref{thm multipliers}.]
We saw in Lemma \ref{lem D mult} that $D$ is a multiplier of $\cA_{S, \varphi}(E)^H$.

Consider the function $h_t$ in \eqref{eq ht}, so $R_t = h_t(D)$. Its Fourier transform is given by \eqref{eq hat ht}, so
$h_t$ has the properties of the function $f$ introduced above Lemma \ref{lem off-diag}. Hence Proposition \ref{prop f(D) multiplier} implies that $R_t$ is a multiplier of $\cA_{S, \varphi}(E)^H$.
\end{proof}

\section{Von Neumann algebras}

\label{sec vN}

We consider the  setting of Subsection \ref{sec prelim index}, and relate the index from Definition \ref{def H index} to an index defined in terms of group von Neumann algebras in the case of cofinite-volume actions.


Let $\Gamma<H$ be a discrete subgroup.
Let $F \subset X$ be a fundamental domain for the action by $\Gamma$, such that the closure of $F$ equals the closure of its interior. Consider the unitary isomorphism
\[
U \colon L^2(E) \to l^2(\Gamma) \otimes L^2(E|_F)
\]
given by
\[
(Us)(\gamma, x) = \gamma s(\gamma^{-1}x),
\]
for $s \in L^2(E)$, $\gamma \in \Gamma$ and $x \in F$.

Let $L$ be the left regular representation of $\Gamma$ in $l^2(\Gamma)$. 
 If $\kappa \in \cA_{\C}(E)^{\Gamma}$ and $\gamma \in \Gamma$, then we write $\gamma \circ \kappa = (\gamma, e) \cdot \kappa \in \cA_{\C}(E)$ in terms of the action \eqref{eq HxH action}, so that
\beq{eq gamma circ kappa}
(\gamma \circ \kappa)(x,x') = \gamma \kappa(\gamma^{-1}x,x'),
\eeq
for $x,x' \in X$.

\begin{lemma}\label{lem weak convergence}
Consider a $\Gamma$-invariant, measurable section $\kappa$ of $E \boxtimes E^*$ that is the Schwartz kernel of a bounded operator on $L^2(E)$, which we also denote by $\kappa$. Then
for all $s,s' \in L^2(E)$,
\beq{eq weak convergence}
\sum_{\gamma \in \Gamma} \bigl(s, U^{-1}  \circ (L_{\gamma} \otimes (\gamma \circ \kappa)|_{F \times F}) \circ Us'
 \bigr)_{L^2(E)} = (s, \kappa s')_{L^2(E)}.
\eeq
\end{lemma}
\begin{proof}
For $\psi \in l^2(\Gamma) \otimes L^2(E|_F)$, $\gamma \in \Gamma$ and $x \in F$,
\[
(U^{-1}\psi)(\gamma x) = \gamma \psi(\gamma^{-1},x).
\]
Using this, and writing out definitions, we find that for all $s' \in L^2(E)$, $\gamma, \gamma' \in \Gamma$ and $x \in F$,
\[
\bigl(U^{-1}  \circ (L_{\gamma} \otimes (\gamma \circ \kappa)|_{F \times F}) \circ Us' \bigr)(\gamma'x) = 
\int_F \gamma' \gamma \kappa(\gamma^{-1}x,y) \gamma^{-1} \gamma'^{-1} s'(\gamma' \gamma y)\, d\mu(y).
\]
By $\Gamma$-invariance of $\kappa$, the right hand side equals
\[
\int_F  \kappa(\gamma'x,\gamma' \gamma y)  s'(\gamma' \gamma y)\, d\mu(y).
\]
So the left hand side of \eqref{eq weak convergence} equals
\[
\sum_{\gamma, \gamma' \in \Gamma} \int_F \int_F
\bigl(s(\gamma'x),  \kappa(\gamma'x,\gamma' \gamma y)  s'(\gamma' \gamma y) \bigr)_E \, d\mu(y)\, d\mu(x),
\]
which equals  the right hand side of \eqref{eq weak convergence}.
\end{proof}


Let $A(E)^{\Gamma}$ be the Banach space of bounded, $\Gamma$-equivariant operators  on $L^2(E)$ with bounded, measurable Schwartz kernels. For $T \in A(E)^{\Gamma}$, with Schwartz kernel $\kappa$, define
\[
\|T\|_{A(E)^{\Gamma}} := \|T\|_{\cB(L^2(E))} + \|\kappa\|_{L^{\infty}(E \boxtimes E^*)}.
\]
Note that $A(E)^{\Gamma}$ is not a Banach algebra. But by Lemma \ref{lem bdd GK}, the algebra $\cA_{\C}(E)^{\Gamma}$ is continuously included in $A(E)^{\Gamma}$.

Conjugation by $U$ is an isometric algebra isomorphism from $\cB(L^2(E))$ onto $ \cB(l^2(\Gamma) \otimes L^2(E|_F))$. It restricts to an injective,  continuous linear map
\[
j_{\cN}\colon A(E)^{\Gamma} \to \cB(l^2(\Gamma) \otimes L^2(E|_F)).
\]
This map is an algebra homomorphism on any algebra included in $A(E)^{\Gamma}$, such as $\cA_{\C}(E)^{\Gamma}$.

If $\calH$ is any Hilbert space, and $A \subset \cB(\calH)$ a $C^*$-algebra, then we write $\cN(\Gamma) \otimes A$ for the algebra of $T \in \cB(l^2(\Gamma) \otimes \calH)$ for which there is a sequence $(T_j)_{j=1}^{\infty}$ in the algebraic tensor product $\C\Gamma \otimes A$ such that for all $\psi, \psi' \in l^2(\Gamma) \otimes \calH$,
\beq{eq weak top}
(\psi, T_j \psi')_{l^2(\Gamma) \otimes \calH} \to (\psi, T \psi')_{l^2(\Gamma) \otimes \calH}.
\eeq
Here we let $\C \Gamma$ act on $l^2(\Gamma)$ via the left regular representation.
\begin{example}
If $\calH = \C = A$, then $\cN(\Gamma) \otimes \C$ is the completion of $\C \Gamma$ in the weak topology on operators on $l^2(\Gamma)$, i.e.\ the von Neumann algebra of $\Gamma$.
\end{example}

\begin{proposition}\label{prop j Cstar}
If $X/\Gamma$ has finite volume, then
\[
j_{\cN}(A(E)^{\Gamma})\subset \cN(\Gamma) \otimes \cK(L^2(E|_F)).
\]
Furthermore, the map
 $j_{\cN}\colon A(E)^{\Gamma} \to \cN(\Gamma) \otimes \cK(L^2(E|_F))$ is continuous.
\end{proposition}
\begin{proof}
Let $\kappa \in A(E)^{\Gamma}$. Then $j_{\cN}(\kappa) \in \cB(l^2(\Gamma) \otimes L^2(E|_F))$, and 
\beq{eq jN bdd}
\|j_{\cN}(\kappa)\|_{\cB(l^2(\Gamma) \otimes L^2(E|_F))} = \|\kappa\|_{\cB(L^2(E))}.
\eeq

Now for every $\gamma \in \Gamma$, the operator $(\gamma \circ \kappa)|_{F \times F}$ on $L^2(E|_F)$ is Hilbert--Schmidt, hence compact, because $F$ has finite volume and $\kappa$ is bounded. So
\[
L_{\gamma} \otimes (\gamma \circ \kappa)|_{F \times F} \in \C\Gamma \otimes \cK(L^2(E|_F)).
\]
By Lemma \ref{lem weak convergence} and unitarity of $U$, the sum 
\[
\sum_{\gamma} L_{\gamma} \otimes (\gamma \circ \kappa)|_{F \times F} 
\]
converges to $j_{\cN}(\kappa)$ in the topology defined by \eqref{eq weak top}. So $j_{\cN}(\kappa)  \in  \cN(\Gamma) \otimes \cK(L^2(E|_F))$.

Continuity of $j_{\cN}$ follows from \eqref{eq jN bdd}: if  $\psi, \psi' \in l^2(\Gamma) \otimes L^2(E|_F)$, then
\[
|(\psi, j_{\cN}(\kappa) \psi')_{l^2(\Gamma) \otimes L^2(E|_F)}| \leq  \|\kappa\|_{\cB(L^2(E))} \|\psi \|_{l^2(\Gamma) \otimes L^2(E|_F)} 
\|\psi' \|_{l^2(\Gamma) \otimes L^2(E|_F)}. 
\]
\end{proof}

Suppose that $X/\Gamma$ has finite volume. Then 
by Lemma \ref{lem bdd GK} and Proposition \ref{prop j Cstar}, we obtain a continuous algebra homomorphism $j_{\cN}\colon \cA_{\C}(E)^{\Gamma}\to \cN(\Gamma) \otimes \cK(L^2(E|_F))$. This induces
\[
(j_{\cN})_*\colon K_0(\cA_{\C}(E)^{\Gamma}) \to K_0( \cN(\Gamma) \otimes \cK(L^2(E|_F))) = K_0(\cN(\Gamma)).
\]
We obtain a link with the index studied in \cite{Atiyah76, Connes82, Wang14}. 
We have
\[
(j_{\cN})_*(\ind_{\C, \Gamma}(D)) \in K_0(\cN(\Gamma)).
\]
On the algebra $\cN(\Gamma)$, we have the von Neumann trace $\tau_e^{\Gamma}$, given by evaluation at $e \in \Gamma$. We will extend this to a trace on a subalgebra of 
$\cN(\Gamma) \otimes \cK(L^2(E|_F))$, and apply it to
$j_{\cN}(\ind_{\C, \Gamma}(D) )$. 
\begin{lemma}\label{lem Tre cts}
For all $\kappa \in \cA_{\C}(E)^{\Gamma}$, the integral
\[
\Tr_e(\kappa) := \int_F \tr(\kappa(x,x))\, d\mu(x)
\]
converges. This defines a continuous trace $\Tr_e$ on $\cA_{\C}(E)^{\Gamma}$.
\end{lemma}
\begin{proof}
This follows from boundedness of kernels in $ \cA_{\C}(E)^{\Gamma}$ and finite volume of $F$.
\end{proof}

For a bounded, continuous $\kappa_F \in \Gamma(E|_F \boxtimes E^*|_F)$, we write
\[
\Tr_F(\kappa_F) = \int_F \tr(\kappa_F(y,y))\, d\mu(y).
\]
It is immediate from the definitions that for all $\kappa \in \cA_{\C}(E)^{\Gamma}$,
\beq{eq tau e Tr e}
(\tau_e^{\Gamma} \otimes \Tr_F)(j_{\cN}(\kappa)) = \Tr_e(\kappa).
\eeq
In particular, Lemma \ref{lem Tre cts} implies that $\tau_e^{\Gamma} \otimes \Tr_F$ defines a trace on $\im(j_{\cN}) \subset \cN(\Gamma) \otimes \cK(L^2(E|_F))$, which is continuous in the topology on $\im(j_{\cN})$ that makes $j_{\cN}$ a homeomorphism.
 Hence we can apply $\tau_e^{\Gamma} \otimes \Tr_F$ to 
\beq{eq index im jN}
(j_{\cN})_*(\ind_{\C, \Gamma}(D)) \in K_0(\im(j_{\cN})).
\eeq

If $D$ is equivariant with respect to $H$, rather than just $\Gamma$, then 
$
\Tr_e(\ind_{\C, \Gamma}(D))
$
equals $\vol(\Gamma \backslash H)$ times the $L^2$-index of $D$, with respect to the action by $H$. This was computed in Theorem 5.2 in \cite{Connes82} for homogeneous spaces, and in \cite{Wang14} in general. We state a consequence for the index \eqref{eq index im jN} in the case where $D$ is a twisted $\Spinc$-Dirac operator. A generalisation can be deduced analogously from Theorem 6.12 in \cite{Wang14}.

Suppose that $X$ has a $H$-equivariant $\Spinc$-structure, with spinor bundle $\calS \to X$ and determinant line bundle $L_{\det}$. Let $E \to M$ be a $H$-equivariant, Hermitian vector bundle, equipped with a $H$-invariant, metric-preserving connection. Let $D$ be the twisted $\Spinc$-Dirac operator on $\calS \otimes E$ associated to a Clifford connection on $\calS$ and the given connection on $E$. Then by \eqref{eq tau e Tr e} and   Theorem 6.10 in \cite{Wang14}, we have the following $L^2$-index theorem:
\[
(\tau^{\Gamma}_e \otimes \Tr_F)\bigl(j_{\cN}(\ind_{\C, \Gamma}(D))\bigr) = \vol(\Gamma \backslash H) \int_F \hat A(X) e^{c_1(L_{\det})/2} \ch(E).
\]

%
%
%

 \bibliographystyle{plain}

\bibliography{mybib}

\end{document}